\documentclass[12pt]{amsart}
\usepackage{amsmath,amstext,amsfonts,amssymb,graphicx,framed,dsfont,enumitem,xcolor,comment}
\usepackage{stmaryrd} 
\makeatletter
\@namedef{subjclassname@2020}{%
  \textup{2020} Mathematics Subject Classification}
\makeatother

\theoremstyle{theorem}
\newtheorem{prop}{Proposition}[section]
\newtheorem{thm}[prop]{Theorem}
\newtheorem{lem}[prop]{Lemma}
\newtheorem{cor}[prop]{Corollary}

\theoremstyle{definition}

\newtheorem{exa}[prop]{Example}
\newtheorem{rem}[prop]{Remark}
\newcommand{\N}{\mathbb{N}}
\newcommand{\R}{\mathbb{R}}
\newcommand{\Z}{\mathbb{Z}}
\newcommand{\Q}{\mathbb{Q}}
\newcommand{\ZZ}{\mathcal{Z}_*}
\newcommand{\ba}{\mathrm{b}}
\newcommand{\bup}{\overline{\mathrm{b}}}
\newcommand{\bdo}{\underline{\mathrm{b}}}
\newcommand{\ua}{\mathrm{u}}
\newcommand{\uup}{\overline{\mathrm{u}}}
\newcommand{\udo}{\underline{\mathrm{u}}}
\newcommand{\da}{\mathrm{d}}
\newcommand{\dup}{\overline{\mathrm{d}}}
\newcommand{\ddo}{\underline{\mathrm{d}}}
\newcommand{\PN}{\mathbb{P}}

\begin{document}

\begin{title}
{\bf  Kneser's theorem for upper Buck density\\ and relative results} 
\end{title}

\date{\today}

\subjclass[2020]{11B05, 11B13}
\keywords{Sumset; Upper Buck density; Inverse theorem}
\maketitle

\centerline{
Fran\c cois HENNECART (\emph{Saint-\'Etienne, France)}
}

\vspace{0.4cm}
\begin{center}
\begin{minipage}[h]{11cm}
{\tiny
Universit\'e Jean Monnet, ICJ UMR5208, CNRS, Ecole Centrale de Lyon, INSA Lyon, Universite Claude Bernard Lyon 1,
42023 Saint-Etienne, France\\
\emph{Email}: \texttt{francois.hennecart@univ-st-etienne.fr}}
\end{minipage}\\[2em]
---------------------------------------
\end{center}

\vspace{0.7cm}
\begin{center}
\begin{minipage}[h]{14cm}
\textsc{Abstract.}
\textrm{\small Kneser's theorem in the integers asserts that denoting by $
\underline{\mathrm{d}}$ the lower asymptotic density, if
$\underline{\mathrm{d}}(X_1+\cdots+X_k)<\sum_{i=1}^k\underline{\mathrm{d}}(X_i)$ then 
the sumset $X_1+\cdots+X_k$ is \emph{periodic} for some positive integer $q$. In this article we establish a similar statement for upper Buck density and compare it with the corresponding result due to Jin involving upper Banach density. We also provide the construction of sequences verifying counterintuitive properties with respect to Buck density of a sequence $A$ and its sumset $A+A$.}
\end{minipage}
\end{center}

\vspace{0.7cm}
\section{\bf Introduction}

We denote by $\Z$ and $\N=\{0,1,2,\dots\}$ the set of all integers and the set of natural integers
respectively. Throughout this paper, for $k\in\N$ and $A,B\subset \Z$, 
$kA$ will denote the set of all multiples $ka$, $a\in A$, while $A+B$ is the set of all sums $a+b$, 
$(a,b)\in A\times B$.
Let  $\ZZ$ be the family of all finite unions of arithmetic progressions
$$
\ZZ=\left\{\bigcup_{i=1}^r(a_i+k_i\N),\ 
r,a_i,k_i\in \N,\ 
1\le i\le r\right\}.
$$
The elements of $\ZZ$ are called \emph{periodic sets} of $\N$. Note that if $F\subset\N$ is finite then
$\N\smallsetminus F=(\{0,1,\dots,2q-1\}\smallsetminus F)+q\N\in\ZZ$ where
$q$ is any integer greater than the largest element of $F$.
We see easily  that $\ZZ$ is closed for finite union, intersection, complement and additive shift by a positive integer.
 Moreover any $A\in\ZZ$ admits a natural density $\da(A)=\lim_{n\to\infty}\frac{|A\cap[1,n]|}{n}$, which is a rational number. Indeed
for some large enough multiple $q$ of $k_1\dots k_r$, we may write
$$
A:=\bigcup_{i=1}^r(a_i+k_i\N)=\bigcup_{j=1}^s(b_j+q\N)
$$
for some integers $s$ and distinct $b_j$ satisfying $0\le b_j<q$, $1\le j\le s$. Thus $\da(A)=\frac{s}{q}$.

\medskip
For $X\subset\N$ we denote by
$$
\bup(X)=\inf\{\da(A),\ X\subset A\in \ZZ\},
$$
its \emph{upper Buck density}, and by
\begin{equation}\label{eqn-1}
\bdo(X)=\sup\{\da(A),\ X\supset A\in \ZZ\}=1-\bup(\N\smallsetminus X),
\end{equation}
its \emph{lower Buck density}.
If $\bup(X)$ and $\bdo(X)$ coincide we say that $X$ has a \emph{Buck density} which is defined by
their common value and denoted by $\ba(X)$. It is the case when 
$X\in\ZZ$ and we have $\ba(X)=\da(X)$. The notion of Buck density has been introduced more than sixty years ago in \cite{Buc1} (see also \cite{Buc2} ).
Pa\v{s}t\'eka reintroduced Buck density in \cite{Pas} and provided 
several general statements that we will use in the present article (see also \cite{Pas2,IPT}).  
 The concept  was recently  generalized in several works by Leonetti and Tringali (see \cite{Leo,Leo2,Leo3,Leo4}). The main feature of (upper and lower) Buck density in comparison with other density functions is its ability to \emph{measure} the distribution of the sets on infinite arithmetic progressions. This fact is naturally relevant with the notion of periodicity  which is naturally used for describing the structure of an infinite set of integers. 

Let $X\subset\N$ and for any $m\ge1$ let 
$\overline{X}^{(m)}$ be the image of $X$ by the natural homomorphism 
$\varphi_m:\Z\longrightarrow \Z/m\Z$. Let $X^{(m)}$ be the set of all integers in $\{0,\dots,m-1\}$ that are remainders in the euclidean division of some $x\in X$ by $m$.
Further we shall denote by $X_{\infty}^{(m)}$ (resp. $X^{(m)}_*$) the set  of integers $k\in \{0,1,\dots,m-1\}$ such that $(k+m\N)\cap X$ is infinite (resp. $(k+m\N)\smallsetminus X$ is finite). We have $X^{(m)}_*\subset X_{\infty}^{(m)}\subset X^{(m)}$
and $|X^{(m)}|=|\overline{X}^{(m)}|$.

Various known results give strong information on the structure of sumsets $Y=X_1+\cdots+X_k$ of integers such that
$f(Y)< \sum_{i=1}^kf(X_i)$ where $f$ is a density function on $\N$, as e.g. $\ddo$, the lower asymptotic density,
$\dup$, the upper asymptotic (or natural) density and $\uup$, the upper uniform density, called also upper Banach density defined for any $X\subset\N$ by
$$
\begin{aligned}
\dup (X)&=\limsup_{n\to\infty}\frac{|X\cap[1,n]|}n,\\
\ddo (X)&=\liminf_{n\to\infty}\frac{|X\cap[1,n]|}n,\\
\uup (X)&=\lim_{n\to\infty}\frac{\max_{k\ge0} |X\cap[k+1,k+n]|}{n},\\
\udo (X)&=\lim_{n\to\infty}\frac{\min_{k\ge0} |X\cap[k+1,k+n]|}{n}.
\end{aligned}
$$
Note that similarly to \eqref{eqn-1} we have the following identities involving complementary sets for their 
asymptotic and uniform densities.
\begin{equation}\label{equdo}
\ddo (X)=1-\dup(\N\smallsetminus X),\quad
\udo (X)=1-\uup(\N\smallsetminus X).
\end{equation}

The story starts in the middle of the twentieth century with the prominent Kneser's theorem for sumsets in $\N$ when $f=\ddo$ is the lower asymptotic density (cf. \cite{Kne}; see also \cite{Hal}).
For upper asymptotic or upper uniform density, several achievements have been made
(cf. \cite{Bih, Bor, Jin2}) in the meantime. In \cite{Gri} Griesmer obtained in arbitrary countable abelian groups a Kneser type statement with  the upper uniform density 
$\uup$  for pairs of sets $(A,B)$ such that $\uup(A+B)< \uup(A)+\uup(B)$. Note also that the case of $\sigma$-finite abelian groups, which is a special case of countable abelian groups, has been considered in \cite{Bie}, yielding a sharper description of sumsets $A+B$ with \emph{small} lower asymptotic density, resp. \emph{small} upper asymptotic density under the restriction $B=\pm A$. 

Our goal is to obtain a Kneser type statement when $f=\bup$ is the upper Buck density in the natural integers.
In Section \ref{S5} we shall see that our results may provide structural information even though Kneser's theorem with 
upper uniform density  (see \cite{Bih,Jin2}) does not apply. 

 We define $\eta(X_1,\dots,X_k)$ by 
\begin{equation}\label{eqeta}
\bup(X_1+\cdots+X_k)=(1-\eta(X_1,\dots,X_k))\sum_{i=1}^k\bup(X_i).
\end{equation}

We say that a subset $A$ of a commutative monoid $G$ with $0$ as identity element, is \emph{periodic} if there exists $g\in G\smallsetminus\{0\}$ such that $A+g\subset A$. When
$G=\N$ this is equivalent to say that $A\in \ZZ$. When $G=\Z/q\Z$ this means that there exists
a proper divisor $d$ of  $q$ such that $A+d\Z/q\Z=A$.

\begin{thm}\label{thmprinc} Let $k\ge2$ and $X_1,\dots,X_k\subset \N$ be non empty and such that 
$$
\bup(X_1+\cdots+X_k)<\sum_{i=1}^k\bup(X_i).
$$
Set $\eta=\eta(X_1,\dots,X_k)>0$
and $\sigma=\sum_{i=1}^k\bup(X_i)$. 
Then there  exist a positive integer $q\le \frac{2k-2}{\eta\sigma}$ 
and, for each $1\le i\le k$,  a periodic set $A_i=\bigcup_{j=1}^{r_i}(a_{ij}+q\N)\in\ZZ$ where the integers $a_{i1},\dots,a_{ir_i}\in\{0,\dots,q-1\}$ are distinct, such that 

\begin{enumerate}[itemsep=2pt,parsep=2pt]
\item $q$ is minimal and $\overline{A}_1^{(q)}+\cdots+\overline{A}_k^{(q)}$ is not periodic
in $\Z/q\Z$.

\item $X_i\subset A_i,\quad \overline{A}_i^{(q)}=\overline{X}_i^{(q)},\quad |\overline{A}_i^{(q)}|=r_i,\quad |\overline{A}_1^{(q)}+\cdots+\overline{A}_k^{(q)}|=\sum_{i=1}^k(r_i-1)+1$.

\item If $r_i\ge2$ for at least two indices $i$'s, then
$\overline{A}_1^{(q)}+\cdots+\overline{A}_k^{(q)}$ is quasi-periodic or an arithmetic progression in $\Z/q\Z$.

\item Moreover for all integer $m\ge1$
and all $(k+1)$-tuples  $(i_1,\dots,i_k,h)\in\prod_{j=1}^k\{1,\dots,r_j\}\times\{0,\dots,m-1\}$, $(a_{1i_1}+\cdots + a_{ki_k}+hq+ mq\N) \cap (X_1+\cdots+X_k) \text{ is infinite}$.

\item Finally
\begin{equation}\label{eqbupX+X}
\bup(X_1+\cdots+X_k)=\frac{\sum_{i=1}^k(r_i-1)+1}q,
\end{equation}
and 
$$
0\le
\frac1k\sum_{i=1}^k\left(\frac{r_i}{q}-\bup(X_i)\right) < \frac{k-1}{kq}.
$$
\end{enumerate}
\end{thm}

We stress the fact that iv) of Theorem \ref{thmprinc} shows that integers
of $\sum_{i=1}^k X_i$ in a same
residue class modulo $q$ \emph{distribute} in an infinite way along all corresponding subclasses modulo $mq$, for any multiple $mq$ of $q$.

An important case concerns small doubling sumsets (with respect to upper  Buck density), namely when $k=2$ and $X_2=X_1$ in the above statement. 
We shall exploit Kemperman's structure theorem (see \cite{Kem,Kem2}) and obtain the following. 
Before we recall that
a non periodic subset $S$ of an abelian group of $G$ is said to be \emph{quasi-periodic} if there exist a non trivial subgroup $K$, an element $s\in S$ and a  non empty set $S''\subsetneq K$ such that $S':=S\smallsetminus(s+S'')$ is $K$-periodic, that is $S'+K=S'$.

\begin{thm}\label{thmsec}
 Let $X\subset\N$ satisfy $\bup(X+X)<2\bup(X)$. Then there exist a positive integer $q$ and a periodic set  $A=\bigcup_{j=1}^{r}(a_{j}+q\N)\in\ZZ$ such that

\begin{enumerate}[itemsep=2pt,parsep=2pt]

\item $q$ is minimal and $\overline{A}^{(q)}+\overline{A}^{(q)}$ is not periodic in $\Z/q\Z$.

\item $X\subset A$, $\overline{X}^{(q)}=\overline{A}^{(q)}$, $|\overline{X}^{(q)}|=r\ge1$, $|\overline{X}^{(q)}+\overline{X}^{(q)}|=2r-1$.

\item If $r=1$ then $\overline{X}^{(q)}=\{\overline{a}\}$ and $\overline{X}^{(q)}+\overline{X}^{(q)}=\{2\overline{a}\}$ for some $\overline{a}\in\Z/q\Z$.

\item If $r\ge2$ then $\overline{X}^{(q)}+\overline{X}^{(q)}$ is

\begin{itemize}

\item either quasi-periodic,

\item  or a non quasi-periodic arithmetic progression in $\Z/q\Z$. In that case
$\overline{X}^{(q)}$ is also  an arithmetic progression in $\Z/q\Z$.

\end{itemize}

\item For any integer $m\ge1$ and all triples $(i,j,h)\in\{1,\dots,r\}^2\times\{0,\dots,m-1\}$, 
$(a_{i}+ a_{j}+hq+ mq\N) \cap (X+X) \text{ is infinite}$.

\item $\bup(X+X)=\frac{2r-1}{q}$ and $\frac{r}{q}-\frac1{2q} <\bup(X)\le\frac rq$.
\end{enumerate}
\end{thm}

The last section is devoted to various remarks and results. In particular we establish the following
(see Propositions \ref{pp62} and \ref{pp67}).

\begin{thm}\label{thm13}
Let $0\le \alpha<\frac14$. Then there exists a set $A\subset \N$ such that
$\ba(A)=\alpha$ and $\bdo(A+A)\ne \bup(A+A)$.
\end{thm}

Ruzsa (unpublished) asked the question whether or not  there exist $0<\nu<1$ and $c>0$ such that
for any $A\subset \N$ having an asymptotic density  
$$
\ddo(A+A)\ge c \cdot (\dup(A+A))^{1-\nu}(\da(A))^{\nu}.
$$
We show that  such statement holds true with Buck density in place of 
asymptotic density and $\nu=\frac12$, $c=1$  (cf. Proposition \ref{pp68}) and
deduce a partial answer to Ruzsa's problem (cf. Corollary \ref{cor69}).

\section{\bf General properties of Buck density and first examples}

A sequence of integers $(m_n)_{n\ge1}$ is said to be \emph{multiplicatively increasing and exhaustive} if $\forall n\ge1, \ m_n\mid m_{n+1}$ and $\forall q\in\N\smallsetminus\{0\}$, $\exists n\ge1$ such that $q\mid m_n$. Then
\begin{equation}\label{eqbupbdo}
\bup(X)=\lim_{n\to\infty}\frac{|{X^{(m_n)}}|}{m_n}=
\lim_{n\to\infty}\frac{|X_{\infty}^{(m_n)}|}{m_n},\quad
\bdo(X)=\lim_{n\to\infty}\frac{|X_*^{(m_n)}|}{m_n}.
\end{equation}
According to these definitions, $m\mapsto |X^{(m)}|$ can be seen as the analogue for upper Buck density of the usual counting function $n\mapsto |X\cap[1,n]|$ which leads to upper asymptotic density
and $n\mapsto \max_{k\ge0} |X\cap[k+1,k+n]|$ to upper Banach density.

\medskip
We have the following properties, which can be found essentially in \cite{Leo,Leo2,Pas}.

\begin{lem}\label{lem21}
 Let $A=\bigcup_{j=1}^r(a_j+q\N)\in \ZZ$ and $X,X'\subset \N$. We have

\begin{enumerate}[itemsep=2pt,parsep=2pt]

\item $\bup(X\cup X')\le \bup(X)+\bup(X')$.

\item If $X\subset A$ and $X'\subset \N\smallsetminus A$ then $\bup(X\cup X')=\bup(X)+\bup(X')$.

\item $\bdo(X\cup X')\le \bdo(X)+\bup(X')$.

\item If $X\subset A$
 then
\begin{equation}\label{eq-2}
\da(A)=1-\da(\N\smallsetminus A) = \bup(X)+\bdo(A\smallsetminus X).
\end{equation}
\end{enumerate} 
\end{lem}

\begin{proof}
\begin{enumerate}[itemsep=2pt,parsep=2pt]
\item  Let $\varepsilon>0$ and  $B,B'\in\ZZ$ such that $X\subset B$, $X'\subset B'$ and 
$$
0\le \da(B)-\bup(X)\le \varepsilon,\quad
0\le \da(B')-\bup(X')\le \varepsilon.
$$
Since $X\cup X'\subset B\cup B'\in \ZZ$ we get 
$$
\bup(X\cup X')\le \da(B\cup B')\le \da(B)+\da(B')\le \bup(X)+\bup(X')+2\varepsilon,
$$ 
and the required inequality follows by letting $\varepsilon$ tend to $0$.

\item 
Let $\varepsilon>0$
and $B\in\ZZ$ such that
$$
X\cup X'\subset B=\bigcup_{j=1}^s(b_j+qm\N)
$$
and 
\begin{equation}\label{eqdb}
0\le \da(B)-\bup(X\cup X')\le \varepsilon.
\end{equation}
Since $X\subset A\cap B\in\ZZ$ and $X'\subset (\N\smallsetminus A)\cap B\in\ZZ$,
we get 
$$
\bup(X)+\bup(X')\le \da(A\cap B)+\da((\N\smallsetminus A)\cap B)=\da(B),
$$
and finally  $\bup(X\cup X')\ge \bup(X)+\bup(X')-\varepsilon$ by \eqref{eqdb}.
We conclude by letting $\varepsilon$ tend to $0$ and by i).

\item Firstly we have  $\bdo(X\cup X')=1-\bup(\N\smallsetminus(X\cup X'))=1-\bup((\N\smallsetminus X)\smallsetminus X')$. 
Secondly by i) we get
$\bup((\N\smallsetminus X)\smallsetminus X')\ge \bup(\N\smallsetminus X)-\bup((\N\smallsetminus X)\cap X')=1-\bdo(X)-\bup((\N\smallsetminus X)\cap X')$. Therefore
$$
\bdo(X\cup X')\le \bdo(X)+\bup((\N\smallsetminus X)\cap X')
$$
and the result follows.

\item By ii) and since $\N\smallsetminus A\in\ZZ$, we infer
$$
\bdo(Y)=1-\bup(\N\smallsetminus Y)=1-(\bup(\N\smallsetminus A)+\bup(X))
=1-(\da(\N\smallsetminus A)+\bup(X)),
$$
and \eqref{eq-2} directly follows from the fact that $\da(\N\smallsetminus A)=1-\da(A)$
(cf. \eqref{equdo}).
\end{enumerate}
\end{proof}

We have the following inequalities implying  $\bup(X)$ and $\bdo(X)$ with respect to the natural (asymptotic) and uniform (Banach) densities (see
\cite{Leo}).

\begin{lem}\label{lem22}
For any $X\subset \N$ we have
\begin{equation}\label{eqn0}
0\le \bdo(X)\le \udo(X)\le \ddo(X)\le \dup(X) \le
\uup(X)\le \bup(X)\le 1.
\end{equation}
\end{lem}

\begin{proof}
For Inequality $\uup(X)\le \bup(X)$ we observe that
if $A\in\ZZ$ then $\uup(A)=\ua(A)=\da(A)$. Hence 
$$
\uup(X)\le \da(A) \quad \text{whenever $X\subset A$,}
$$
yielding $\uup(X)\le \bup(X)$, as required. 
In view of their definition Inequality $\dup(X) \le\uup(X)$ is obvious. 
By \eqref{eqn-1} and \eqref{equdo} we obtain the reversed inequalities 
for lower densities.
\end{proof}

\begin{rem}\label{rm23}
If $\dup_{\Phi}(\cdot)$ denotes the upper density along a F\o lner sequence
$\Phi=(F_n)$ we obtain $\dup_{\Phi}(X)\le \uup(X)\le \bup(X)$ since
$$
\uup(X)=\sup\{\dup_{\Phi}(X),\ \Phi\text{ is a F\o lner sequence}\}.$$

\end{rem}

The following example is a particular case of \cite[Corollary 3.4]{Leo2}.

\begin{exa}\label{exa24}
For any natural number $t\ge1$ let 
$$
\PN_t=\{n\in \N\smallsetminus\{0,1\} \text{ such that } \omega(n)\le t\}
$$
where $\omega(n)=\displaystyle\sum_{\substack{p\text{ prime}\\p\mid n}}1$.
We thus have $\PN_1=\PN$ the set of all prime numbers. 

Let $k\ge1$. If 
$$
\PN_t\subset \bigcup_{a=1}^k(a+k\N)
$$
then 
$$
\PN_t\subset \bigcup_{\substack{a=1\\ \omega(\gcd(a,k))\le t}}^k(a+k\N).
$$
We denote $\varphi_t(k)=|\{1\le a\le k : \omega(\gcd(a,k))\le t\}|$. Thus $\varphi_0=\varphi$ is the Euler totient function and we have
$$
\bup(\PN_t)\le\frac{\varphi_t(k)}{k}.
$$
We see that
$$
\varphi_t(k)=\sum_{\substack{d\mid k\\\omega(d)\le t}}\sum_{\substack{a=1\\\gcd(a,k)=d}}^k1
=\sum_{\substack{d\mid k\\\omega(d)\le t}}
\sum_{\substack{a=1\\\gcd(a,k)=1}}^{k/d}1
=\sum_{\substack{d\mid k\\\omega(d)\le t}}\varphi\left(\frac kd\right).
$$
Denote by $(p_i)_{i\ge1}$ the increasing sequence of prime numbers and assume that $(k_r)$ is the  increasing sequence of positive integers defined by
$$
k_r=\prod_{i=1}^rp_i.
$$
Since $k_r$ is squarefree we have by Mertens' second theorem
$$
\varphi_t(k_r)=\sum_{\substack{d\mid k_r\\\omega(d)\le t}}
\frac{\varphi(k_r)}{\varphi(d)}\le
\frac{\varphi(k_r)}{t!}\left(
\sum_{i=1}^r\frac1{p_i-1}\right)^t\ll
\varphi(k_r)\frac{(\ln\ln r)^t}{t!}
$$
where the implied constant is absolute.
By Mertens' third theorem $\varphi(k_r)\ll \frac{k_r}{\ln r}$ hence 
$\bup(\PN_t)\ll \frac{(\ln\ln r)^t}{t!\ln r}$ for any $r$. We infer  $\ba(\PN_t)=
\bup(\PN_t)=0$.

\end{exa}

\begin{exa}\label{exa26} 
Let $(k_r)_{r\ge1}$ be a sequence of positive integers such that
$\forall d\in\N\smallsetminus\{0\}$, $\exists r$ such that $d\mid k_1\dots k_r$.
We set $K_r=k_1\dots k_r$.
Let $X=\{r+K_r,\ r\ge1\}\subset\N$. Let $m\ge1$ and $r$ be large enough so that $m\mid K_r$.
The finite set
$$
\{n+m+K_{n+m} : n=r,\dots,r+m-1\}
$$
intersects all residue classes modulo $m$. It follows that if
$X\subset R+m\N$ then $R$ necessarily contains a representative of each residue classes modulo $m$. Hence $\bup(X)=1$. Plainly 
$\bdo(X)=0$ and $\ua(X)=0$. 
\end{exa}

\begin{exa}
The sum of digits in base $2$ of $n=\sum_{k\ge0}\epsilon_k2^k$ with $\epsilon_k\in\{0,1\}$,  is defined by $s_2(n)=\sum_{k\ge0}\epsilon_k$. In average its value is $\frac{\ln n}{\ln4}$.
For any map $f:\N\rightarrow \R_+$ satisfying $f(n)=o(\ln n)$ and $f(n)\to\infty$ as  $n\to\infty$, the sequence 
$$
A=\left\{n\in\N : s_2(n)\le f(n)\right\}
$$
satisfies $\da(A)=0$ whence $\bdo(A)=\udo(A)=0$. 
Furthermore for any integer $k$ there exists $N$ such that
$f(n)\ge 2+\frac{\ln k}{\ln 2}$ for any $2^N\le n\le 2^N+k$. 
For those $n$ we have $s_2(n-2^N)\le 1+\left\lfloor \frac{\ln k}{\ln 2}\right\rfloor$,
thus $s_2(n)\le f(n)$ that is $n\in A$.
We infer $\uup(A)=1$ hence
$\bup(A)=1$ by \eqref{eqn0}.
\end{exa}

\begin{exa}\label{exa27}
Let $X_0=\left\{\sum_{k\in K}4^k : K\subset\N\text{ finite}\right\}$. Then for any $m\ge1$ we have
$$
X_0\subset \left\{\sum_{k=0}^{m-1}\epsilon_k4^k :\epsilon_k\in\{0,1\}\quad
(0\le k<m)\right\}+4^m\N,
$$
thus $\bup(X_0)\le \frac{2^m}{4^m}$. This yields $\ba(X_0)=0$. Now
$$
X_0+X_0=\left\{\sum_{k\ge0}\epsilon_k2^k\in\N : \forall k,\ \epsilon_{k}\in\{0,1\}\text{ and }\Big(\epsilon_{2k}=1\implies \epsilon_{2k+1}=0\Big)\right\}
$$
thus 
$$
X_0+X_0=\bigcap_{r\ge0}\left\{\sum_{k\ge0}\epsilon_k2^k\in\N : \forall k,\ \epsilon_{k}\in\{0,1\}\text{ and }(\epsilon_{2r},\epsilon_{2r+1})\in\{(0,0),(0,1),(1,0)\}\right\}.
$$
This gives $\bup(X_0+X_0)\le \left(\frac34\right)^m$ for any $m$, thus $\ba(X_0+X_0)=0$.
\end{exa}

\begin{exa}\label{exa28}
Let $0<\alpha<1$ be a real number and $\alpha=0.a_1a_2\dots$ its \emph{regular} dyadic expansion, that is
$a_n=\lfloor 2^n\alpha\rfloor\bmod2$, $n\ge1$.
Let $\mathcal{M}(\alpha)=\{n\ge1 : a_n=0\}$. Note that $\mathcal{M}(\alpha)$ is infinite.
Set
$$
B_{\alpha}=\bigcup_{\substack{n\ge1\\ a_n=1}}\left(2^{n-1}+2^n\N\right).
$$
The Buck density of this sequence have been fixed in \cite{Leo}. For sake of completeness we give some details of the above and extend its study in the sumset framework.
Since 
$$
\bigcup_{\substack{1\le n\le m\\ a_n=1}}\left(2^{n-1}+2^n\N\right)\subset B_{\alpha}
$$ 
we get $\bdo(B_{\alpha})\ge0.a_1a_2\dots a_m\ge\alpha-\frac1{2^m}$ for any $m$. Therefore $\bdo(B_{\alpha})\ge\alpha$. For any $m\in \mathcal{M}(\alpha)$, we also
have 
$$
B_{\alpha}\subset \bigcup_{\substack{1\le n\le m-1\\ a_n=1}}\left(2^{n-1}+2^n\N\right)
\cup(2^{m-1}+2^m\N)
$$ 
thus $\bup(B_{\alpha})\le 0.a_1a_2\dots a_{m-1}1\le \alpha+\frac1{2^m}$.
We infer  $\bup(B_{\alpha})\le \alpha$ yielding $\ba(B_{\alpha})=\alpha$.\\[0.3em]
Now let  $r=\min\{n\ge1 : a_n=1\}$. We plainly have 
\begin{equation}\label{Balpha}
B_{\alpha},B_{\alpha}+B_{\alpha}\subset 2^{r-1}\N.
\end{equation}
More precisely if $\alpha=2^{-r}$ then 
$B_{\alpha}+B_{\alpha}=2^r\N$. Otherwise let $s=\min\{n>r : a_n=1\}$. We have $2^{s-1}\in B_{\alpha}$ and we easily infer 
\begin{equation*}
B_{\alpha}+B_{\alpha}=2^r\N\cup(2^{r-1}+2^{s-1}+2^r\N).
\end{equation*}
Therefore $\ba(B_{\alpha}+B_{\alpha})=2^{-\lfloor\log_2(\alpha^{-1})\rfloor}<2\alpha=2\ba(B_{\alpha})$. This means that $B_{\alpha}$ fulfills a \textsl{small doubling condition} with respect to Buck density, but also with respect to uniform density as well as to asymptotic density by \eqref{eqn0}. We shall use the sequences $B_{\alpha}$ in \S\ref{SS65}.
\end{exa}

\begin{exa}\label{exa29}
Let $0<\alpha,\beta,\gamma<1$ be three real numbers and
$(N_k)_{k\ge1}$ be an increasing sequence of positive numbers such that
$N_{k}=o(N_{k+1})$ and $2^{k}=o(N_k)$ as $k\to\infty$.  Set $r_k=\lfloor \alpha 2^k\rfloor$, $k\ge1$. We thus have $r_{k+1}=2r_k$ or $2r_{k}+1$ and there exists $R_k\subset\{0,1,\dots,2^k-1\}$ such that $|R_k|=r_k$ and $R_k+2^k\N\subset R_{k+1}+2^{k+1}\N$.
Moreover $r_k\sim \alpha 2^k$ as $k\to\infty$.
Finally let $\theta$ be a positive irrational number and define 
$$
A=\bigcup_{k\ge1}A_k\quad \text{with}\quad A_k=\{N_k\le n\le (1-\gamma)^{-1}N_{k}\,\mid\,
\{\theta n\}<\beta\text{ and }(n-R_k)\cap 2^k\N\ne\varnothing\}.
$$
Since the sequence $(\theta n)$ is equidistributed modulo 1 we get $|A_k|=(1+o(1))\beta \frac{r_k}{2^k}((1-\gamma)^{-1}-1)N_k$, $k\to\infty$, hence
$$
\frac{|A\cap[1,(1-\gamma)^{-1}N_k]|}{(1-\gamma)^{-1}N_k}=
\frac{|A_k|+o(N_k)}{(1-\gamma)^{-1}N_k}\sim
\alpha\beta\gamma,\quad k\to\infty.
$$
It is rather clear that
\begin{align*}
\dup(A)&=\lim_{k\to\infty}\frac{|A\cap[1,(1-\gamma)^{-1}N_k]|}{(1-\gamma)^{-1}N_k}=\lim_{k\to\infty}\frac{|A_k|+o(N_k)}{(1-\gamma)^{-1}N_k},\\
\uup(A)&=\lim_{k\to\infty}\frac{|A_k|}{((1-\gamma)^{-1}-1)N_k},
\end{align*}
and since members of $A_k$ are well-distributed 
in each modulo class $a+2^k\N$, $a\in R_k$,
$$
\bup(A)=\lim_{k\to\infty}\frac{r_k}{2^k}.
$$
Therefore  $\dup(A)=\alpha\beta\gamma$, $\uup(A)=\alpha\beta$ and $\bup(A)=\alpha$. 
\end{exa}

We give below a useful characterisation of  sets $X$ having maximal upper Buck  density in a given periodic set $A\in\ZZ$. 

 \begin{prop}\label{pp210}
Let $X\subset A:=\bigcup_{j=1}^r(a_j+q\N)\in\ZZ$ where the $a_j$'s are distinct modulo $q$. Then $\bup(X)=\da(A)=\frac{r}{q}$ if and only if
\begin{equation}\label{eq-1}
\forall\, m\ge1,\ \forall\, 1\le j\le r,\ \forall\, 0\le k\le m-1,\ 
\exists\, \ell\in\N\text{ such that } a_j+kq+\ell mq \in X.
\end{equation}
\end{prop}

\begin{proof}
If \eqref{eq-1} is not satisfied then there exist $m\ge1$, $1\le j_0\le r$ and $0\le k_0\le m-1$
such that $a_{j_0}+k_0q+\ell mq\not\in X$ for any $\ell\in\N$.
Thus
$$
X\subset A':=\bigcup_{\substack{1\le j\le r\\0\le k\le m-1\\(j,k)\ne(j_0,k_0)}}(a_j+kq+mq\N)
$$
We infer 
$\bup(X)\le \da(A')=\frac{rm-1}{mq}<\frac rq$.

Conversely assume that \eqref{eq-1} is satisfied. Let $q'\ge1$ and $0\le a'_1<a'_2<\cdots<a'_{r'}<q'$ such that
$$
X\subset\bigcup_{j=1}^{r'}(a'_j+q'\N).
$$
Then the set
$$
E:=\left\{a_j+kq,\ 1\le j\le r,\ 0\le k<\frac{q'}{\gcd(q,q')}\right\}
$$ 
satisfies  ${E}^{(q')}\subset\{a'_j,\ 1\le j\le r'\}$ which has cardinality $r'$. Let $a_i+kq\in E$ be fixed. If $(j,h)$ is such that
$$
1\le j\le r,\ 0\le h<\frac{q'}{(q,q')} \text{ and } a_i+kq\equiv a_j+hq\bmod q'
$$
then $a_j\equiv a_i$ $\bmod (q,q')$. There are at most $\frac{q}{(q,q')}$ such 
$a_j$'s since they are all distinct modulo~$q$. Once such $a_j$ is fixed, we get
$$
(h-k)\frac{q}{\gcd(q,q')}\equiv \frac{a_i-a_j}{\gcd(q,q')} \bmod{\frac{q'}{\gcd(q,q')}},
$$
yielding only one occurrence for $0\le  h<\frac{q'}{\gcd(q,q')} $. Thus
$$
r'=|{E}^{(q')}|\ge\frac{|E|}{\frac{q}{\gcd(q,q')}}=
\frac{\frac{rq'}{\gcd(q,q')}}{\frac{q}{\gcd(q,q')}}=\frac{rq'}{q}.
$$
We infer $\frac{r'}{q'}\ge\frac{r}{q}$ and we may conclude that $\bup(X)=\frac rq$. 
\end{proof}

Note that in necessary and sufficient Condition \eqref{eq-1} we may change without any loss `$\exists\ell\in\N$' into 
`$\exists\text{ infinitely many }\ell\in\N$'.

\begin{exa}\label{exa211}
When $X\subset a+q\N$ satisfies 
$$
\forall m\ge1,\quad \left|\{0\le h\le m-1,\ |X\cap(a+hq+mq\N)|=\infty\}\right|>\frac m2,
$$
 then  $\forall m \ \forall h\ |(X+X)\cap(2a+hq+mq\N)|=\infty$,
hence
$\bup(X+X)=\frac1q$ by Proposition~\ref{pp210}. This means that 
$X+X$ \emph{spread out} in all arithmetic progressions included in $2a+q\N$. Theorem~\ref{thmprinc} shows that this picture is somewhat generic.
\end{exa}

\section{\bf Kneser's theorem for upper Buck density -- Proof of Theorem \ref{thmprinc}}

Ths section is devoted to the proof of Theorem \ref{thmprinc}. We shall obtain it
gradually by proving several intermediate statements. We start by a lemma that
will allow us, in the subsequent lemma, to use Kneser's theorem for \emph{small} sumsets in an abelian group.

\begin{lem}\label{lem31}
Let $k\ge2$, $X_i\subset \N$, $1\le i\le k$, be such that $\bup(X_1+\cdots+X_k)<\sum_{i=1}^k\bup(X_i)$ and set $\eta=\eta(X_1,\dots,X_k)>0$.
Then for any $0<\varepsilon < \eta$, 
 there exists a modulus $m$ such that 
\begin{equation}\label{eq1}
\bup(X_1+\cdots+X_k)\le  \frac{|\overline{X}_1^{(m)}+\cdots+\overline{X}_k^{(m)}|}{m}\le \bup(X_1+\cdots+X_k)+\varepsilon\sum_{i=1}^k\bup(X_i), 
\end{equation}
and
\begin{equation}\label{eq2}
|\overline{X}_1^{(m)}+\cdots+\overline{X}_k^{(m)}|\le (1-\eta+\varepsilon)
\sum_{i=1}^k|\overline{X}_i^{(m)}|<\sum_{i=1}^k|\overline{X}_i^{(m)}|.
\end{equation}
\end{lem}

\begin{proof}
There exists $B:=\bigcup_{j=1}^s (b_j+m\N)\in\ZZ$
where the $b_j$'s are distinct modulo $m$, 
 such that 
\begin{align*}
&X_1+\cdots+X_k\subset B,\\
&\overline{X}_1^{(m)}+\cdots+\overline{X}_k^{(m)}= \overline{B}^{(m)},\\ 
&\da(B)=\frac{|\overline{B}^{(m)}|}m,\\
&\bup(X_1+\cdots+X_k)\le \da(B)\le  \bup(X_1+\cdots+X_k)+\varepsilon\sum_{i=1}^k\bup(X_i),
\end{align*}
which implies \eqref{eq1}.
Notice that by letting $s=|\overline{X}_1^{(m)}+\cdots+\overline{X}_k^{(m)}|$ we have $\bup(X_1+\cdots+X_k)\le \da(B)=\frac sm$. Moreover 
for each $1\le i\le k$, 
 there exist distinct residues modulo $m$, $a_{ij}$, $1\le j\le r_i=|\overline{X}_i^{(m)}|$, such that
\begin{align*}
&X\subset A_i:=\bigcup_{j=1}^{r_i} (a_{ij}+m\N),\\
&\overline{X}_i^{(m)}=\overline{A}_i^{(m)}.
\end{align*}
We thus have  $\bup(X_i)\le \da(A_i)=\frac{r_i}{m}$.
This gives
$$
\frac sm \le   \bup(X_1+\cdots+X_k)+\varepsilon\sum_{i=1}^k\bup(X_i)=(1-\eta+\varepsilon)\sum_{i=1}^k\bup(X_i)\le (1-\eta+\varepsilon)\sum_{i=1}^k\frac{r_i}{m},
$$
yielding \eqref{eq2}.
\end{proof}
We shall apply the following version of Kneser's theorem for small finite sumsets in an abelian group completed with Kemperman's qualitative structure result (cf. Theorem~2.1 of \cite{Kem}). 

\begin{lem}\label{lem32}
Let $\eta>0$, $G$ be an abelian group and $S_1,\dots,S_k$ be non empty finite subsets of $G$ such that
$|S_1+\cdots+S_k|\le (1-\eta)\sum_{i=1}^k|S_i|$. 
\begin{enumerate}
\item Then there exists a maximal subgroup $H$ of $G$ and positive integers $r_1,\dots,r_k$
such that
\begin{align*}
&S_1+\cdots+S_k=S_1+\cdots+S_k+H,\\
&|S_i+H|=r_i|H|,\ 1\le i\le k,\\
&|S_1+\cdots+S_k|=\sum_{i=1}^k(|S_i+H|-|H|)+|H|=\left(\sum_{i=1}^k(r_i-1)+1\right)|H|,\\
&\sum_{i=1}^kr_i\le \frac{k-1}{\eta}
\quad\text{and}\quad
0\le \sum_{i=1}^k (r_i|H| -|S_i|)<(k-1)|H|.
\end{align*}
\item 
Moreover if there are at least two $i$'s such that $r_i\ge2$, then $(S_1+\cdots+S_k)/H$ is quasi-periodic or an arithmetic progression in the factor group $G/H$.
\end{enumerate}
\end{lem}

\begin{proof}
In i), we only need to notice that the fourth statement follows from the third. Indeed the latter gives
$$
\left(\sum_{i=1}^k(r_i-1)+1\right)|H|\le (1-\eta)\sum_{i=1}^k|S_i|
\le (1-\eta)\sum_{i=1}^kr_i|H|
$$
hence $k-1\ge \eta \sum_{i=1}^kr_i$. Statement ii) follows from  Theorem 2.1 of (\cite{Kem}), which gives the desired result when $k=2$. When $k\ge3$ we may assume that $r_1,r_2\ge2$. We let $T=S_2+\cdots+S_k$. Since $H$ is maximal, $T/H$ is not periodic thus we must have
$|T+H|\ge\sum_{i=2}^k(|S_i+H|-|H|)+|H|$ by Kneser's theorem. We infer
\begin{align*}
|(S_1+H)+(T+H)|&=|S_1+T+H|=\sum_{i=1}^k|(|S_i+H|-|H|)+|H|\\
&=|S_1+H|+\sum_{i=2}^k|(|S_i+H|-|H|)\\
&\le |S_1+H|+|T+H|-|H|.
\end{align*} 
But by Kneser's theorem again, the last inequality is an equality since $(S_1+T+H)/H$ is not periodic in $G/H$. We get 
$|S_1/H+T/H|=|S_1/H|+|T/H|-1$ where $|S_1/H|=r_1\ge2$ and $|T/H|\ge|S_2/H|=r_2\ge2$.
and we conclude by Kemperman's theorem that $(S_1+T)/H=(S_1+\cdots+S_k)/H$ is either quasi-periodic or an arithmetic progression in $G/H$.
\end{proof}

\begin{prop}\label{pp33}
Let $k\ge2$ and $X_1,\dots,X_k\subset \N$ be non empty sets of non negative  integers such that $\bup(X_1+\cdots+X_k)<\sum_{i=1}^k\bup(X_i)$. Set $\eta=\eta(X_1,\dots,X_k)$
and $\sigma=\sum_{i=1}^k\bup(X_i)$.\\
Then for any $0<\varepsilon
<\eta$, there exists a positive integer $q$ satisfying 
\begin{equation}\label{eqqle}
q\le \frac{k-1}{(\eta-\varepsilon) \sigma},
\end{equation}
such that
$$
|\overline{X}_1^{(q)}+\cdots+\overline{X}_k^{(q)}|=\sum_{i=1}^k|\overline{X}_i^{(q)}|-(k-1)
$$
and
\begin{equation}\label{eqbdo}
\bdo((X_1^{(q)}+\cdots+X_k^{(q)}+q\N)\smallsetminus(X_1+\cdots+X_k))\le \varepsilon\sigma.
\end{equation}
Moreover if there is at least two $i$'s such that $|X_i^{(q)}|\ge2$, then
$\overline{X}_1^{(q)}+\cdots+\overline{X}_k^{(q)}$  is quasi-periodic or an arithmetic progression in $\Z/q\Z$.
\end{prop}

\begin{proof}
By Lemma \ref{lem31}, there exists an  integer $m$ such that the sets $\overline{X}_i^{(m)}$, $1\le i\le k$, satisfy the small sumset hypothesis of Lemma \ref{lem32}. Hence
$\overline{X}_1^{(m)}+\cdots+\overline{X}_k^{(m)}$ is periodic in $\Z/m\Z$ with 
period $H=q\Z/m\Z$ for some $q\mid m$. Furthermore there are positive integers $r_i$ such that $|\overline{X}_i^{(m)}+H|=r_i|H|$ ($1\le i\le k$) and 
$|\overline{X}_1^{(m)}+\cdots+\overline{X}_k^{(m)}|=\left(\sum_{i=1}^k(r_i-1)+1\right)|H|$. We get $|\overline{X}_1^{(q)}+\cdots+\overline{X}_k^{(q)}|=\sum_{i=1}^k(r_i-1)+1$ and
$$
\overline{X}_1^{(m)}+\cdots+\overline{X}_k^{(m)}+m\N=\overline{X}_1^{(q)}+\cdots+\overline{X}_k^{(q)}+q\N.
$$
Thus
\begin{multline}\label{eqbupXX}
\bup(X_1+\cdots+X_k)\le \frac{|\overline{X}_1^{(m)}+\cdots+\overline{X}_k^{(m)}|}{m}=
\frac{|\overline{X}_1^{(q)}+\cdots+\overline{X}_k^{(q)}|}{q}=\frac{\sum_{i=1}^k(r_i-1)+1}{q}\\
\le \bup(X_1+\cdots+X_k)+\varepsilon\sigma,
\end{multline}
yielding
\begin{multline*}
\bup(X_1+\cdots+X_k)+\varepsilon\sigma\ge 
\frac{\sum_{i=1}^k(r_i-1)+1}{q}
\ge\sum_{i=1}^k\bup(X_i)-\frac{k-1}q
\\=\bup(X_1+\cdots+X_k)+\eta\sigma-\frac{k-1}q,
\end{multline*}
which implies \eqref{eqqle}.
Bound \eqref{eqbdo} follows from the upper bound in \eqref{eqbupXX} and \eqref{eq-2}.
\end{proof}

Bound \eqref{eqbdo} reveals that the relative upper Buck density of $X_1+\cdots+X_k$ in
$X_1^{(q)}+\cdots+X_k^{(q)}+q\N\in\ZZ$ is close to $1$. By \eqref{eqbupbdo} we obtain the following structure result.

\begin{thm}\label{thm34}
 Let $k\ge2$ and $X_1,\dots,X_k\subset \N$ be
such that $\bup(X_1+\cdots+X_k)<\sum_{i=1}^k\bup(X_i)$. 
Set $\eta=\eta(X_1,\dots,X_k)$ and $\sigma=\sum_{i=1}^k\bup(X_i)$. 
Let $\varepsilon$ be a real number such that $0<\varepsilon\le \frac{\eta}2$. Then there 
exist a positive integer $q=q(\varepsilon)$ such that
\begin{equation}\label{eqqle2}
q\le \frac{k-1}{(\eta-\varepsilon) \sigma}\le \frac{2k-2}{\eta
\sigma}
\end{equation}
and a periodic set $B=\bigcup_{j=1}^{s}(b_j+q\N)\in\ZZ$ 
where $s=\sum_{i=1}^k|\overline{X}_i^{(q)}|-(k-1)$,  such 
that 
\begin{enumerate}[itemsep=2pt,parsep=2pt]

\item $q$ is minimal and $\overline{B}^{(q)}$ is not periodic in $\Z/q\Z$,

\item  $X_1+\cdots+X_k\subset B$,

\item if there is at least two $i$'s such that $|\overline{X}_i^{(q)}|\ge2$ then
$\overline{X}_1^{(q)}+\cdots+\overline{X}_k^{(q)}=\overline{B}^{(q)}$  is quasi-periodic or an arithmetic progression in $\Z/q\Z$.

\item Furthermore for all integer $m\ge1$, 
there are at most $\varepsilon sm$ couples $(j,h)\in\{1,\dots,s\}\times\{0,\dots,m-1\}$
such that  $(b_j+hq+ mq\N) \cap (X_1+\cdots+X_k) \text{ is finite}.$
\end{enumerate}
\end{thm}

Now we derive Theorem \ref{thmprinc}.

\begin{proof}[Proof of Theorem \ref{thmprinc}]
We define
$$
\mathcal{Z}_{\infty}=\left\{\bigcup_{j=1}^t(c_j+mE_j),\ 
t,m,c_j\in \N,\ 
E_j\subset\N\text{ is infinite},\ 
0\le c_j\le m-1,\ 
1\le j\le t\right\}.
$$
Theorem \ref{thm34} implies that under the same hypotheses and notation, there exist  sets $B\in\ZZ$, $B'\in \mathcal{Z}_{\infty}$ and an integer $q\le\frac{2k-2}{\eta\sigma}$ 
such that $\overline{B}^{(q)}=\overline{B'}^{(q)}$,
$B'\subset X_1+\cdots+X_k\subset B$ and $|\overline{B}_{mq}\smallsetminus \overline{B'}_{mq}| \le  \varepsilon\sum_{i=1}^k|\overline{X}_i^{(mq)}|$ for any $m\ge1$.

By \eqref{eqqle2}, there exists an increasing sequence of natural numbers
$(h_n)_{n\ge1}$ and a positive  integer $q$ such that $q(\frac1{h_n})=q$ for any $n\ge1$. Hence for any $m\ge1$, we can find $n$ such that $\sum_{i=1}^k|\overline{X}_i^{(mq)}|<h_n$. We infer
$|\overline{B}^{(mq)}\smallsetminus \overline{B'}^{(mq)}| < 1$ hence
 $\overline{B}^{(mq)}=\overline{B'}^{(mq)}$ as required. Finally 
\eqref{eqbupX+X} follows from Proposition \ref{pp210}.
\end{proof}

\section{\bf Small doubling sets -- Proof of Theorem \ref{thmsec}}

If $S$ is an arithmetic progression in $\Z/m\Z$ then $S+S$ obviously does. The 
next lemma provides a kind of reciprocal statement. It is certainly known but we give a proof 
for sake of completeness. Before we notice that
a quasi-periodic proper subset of $\Z/m\Z$ is an arithmetic progression if
and only it is the complement set of a singleton. Hence if $S+S$ is both 
quasi-periodic and an arithmetic progression then $S+S=\Z/m\Z\smallsetminus\{a\}$ for some $a$. Any $S$ such that $a-S=\Z/m\Z\smallsetminus S$ satisfies
$|S|=m-|S|$, hence $m$ is even and $|S+S|=m-1=2|S|-1$.

\begin{lem}\label{lem41}
 Let  $S\subset \Z/m\Z$ such that $|S+S|=2|S|-1$ and assume that $S+S$ is not quasi-periodic. 
Then $S$ is an arithmetic progression in  $\Z/m\Z$.
\end{lem}

\begin{proof}
Set $n=|S|$.
If $n=1$ then $S$ is plainly an arithmetic progression. Assume that 
\begin{equation}\label{nge2}
n\ge2.
\end{equation}
By Kemperman's theorem (cf. \cite{Kem}), $S+S$ is an arithmetic progression since not quasi-periodic.
Changing if necessary $S$ by $S-s$ for some $s\in S$ we may freely assume that 
$0\in S$ hence $0\in S+S$. This implies that for some $\varphi_m(g)\in\Z/m\Z$
or order at least equal to $2n-1$
$$
S+S=\{\varphi_m(ig),\ -(2n-t-1)\le i\le t-1\}
\quad\text{for some $1\le t\le 2n-1$.}
$$
Clearly $S\subset g\Z/m\Z$, hence we may assume that $\gcd(m,g)=1$.
Otherwise we change $S$ into $\frac{A}{\gcd(m,g)}$ which can be seen as a subset of $\Z/m_1\Z$  where $m_1=\frac{m}{\gcd(m,g)}$. Finally changing $S$ into
$g^{-1}S$ where $g^{-1}$ is the inverse of
$g$ modulo $m$, we may assume that $g=1$. Hence
$$
S+S=\varphi_m(\Z\cap[-2n+t+1, t-1]).
$$
This yields
\begin{equation}\label{eqScap}
S\cap (-S+\varphi_m(\Z\cap [t,t+m-2n]))=\varnothing.
\end{equation}
Furthermore since $S+S$ is not quasi-periodic we have $2n-1\le m-2$ thus
\begin{equation}\label{eqm-2n}
m-2n\ge1.
\end{equation}
Assume by contradiction that $S$ is not an arithmetic progression with difference $1$ in $\Z/m\Z$. Hence we may write $S=\{\varphi_m(s_i),\ 1\le i\le n\}\subset\varphi_m(\Z\cap[s_1,s_n])$ where $n\le s_n-s_1\le (t-1)-(-2n+t+1)=
 2n-2$ since $S\subset S+S$.
It follows that 
\begin{align*}
|-S+\varphi_m(\Z\cap[t,t+m-2n])|&=|\{\varphi_m(\{s_n-s_i,\ n\ge i\ge1\}) +\varphi_m(\Z\cap[0,m-2n])|\\
&=|\varphi_m(\{s_n-s_i,\ n\ge i\ge1\}+\Z\cap[0,m-2n])|\\
&=|\{s_n-s_i,\ n\ge i\ge1\}+\Z\cap[0,m-2n]|,
\end{align*}
since  $\{s_n-s_i,\ n\ge i\ge1\}\subset\Z\cap[0,2n-2]$. Moreover 
$\{s_n-s_i,\ n\ge i\ge1\}$ is not an arithmetic progression in $\Z$. Therefore 
by \eqref{nge2} and \eqref{eqm-2n} we get
\begin{align*}
|\{s_n-s_i,\ n\ge i\ge1\}+\Z\cap[0,m-2n]|&\ge
|\{s_n-s_i,\ n\ge i\ge1\}|+|\Z\cap[0,m-2n]|\\
&\ge
n+m-2n+1=m-n+1.
\end{align*}
We infer $|-S+\varphi_m(\Z\cap[t,t+m-2n])|\ge m-n+1$,
which contradicts \eqref{eqScap}.
\end{proof}

Theorem \ref{thmsec} follows from Theorem \ref{thmprinc} and 
the above lemma which provides iv).

\section{\bf Banach density versus Buck density}
\label{S5}

\subsection{} Theorem 1.4 of \cite{Jin2} shows  that if 
$\uup(A+B)<\uup(A)+\uup(B)$ then 
\begin{equation}\label{eqbdbd}
\bigcup_{n\ge1}\left((S+q\N)\cap[c_n,d_n]\right)\subset A+B\subset S+q\N
\end{equation}
for some positive integer $q$, some non empty set $S\subset \{0,1,\dots,q-1\}$ and some increasing sequences $(c_n)$, $(d_n)$ satisfying 
$\lim_{n\to\infty}(d_n-c_n)=\infty$. Hence 
$\bup(A+B)=\frac{|S|}{q}=\uup(A+B)$. We get
$$
\bup(A+B)<\uup(A)+\uup(B)\le 
\bup(A)+\bup(B).
$$
This shows that any small sumset pair $A,B$ for upper Banach density is
a small sumset pair for upper Buck density. The converse is obviously false: let 
$B_0$ be any additive basis of order~$2$, i.e. $B_0+B_0=\N$, with zero Banach density and set
$A=B=B_0\cup\{n!+n,\ n\ge0\}$. Then $A+A=\N$ and
$$
\bdo(A)=\ua(A)=0,\ \bup(A)=1\quad \text{and}\quad \ua(A+A)=\ba(A+A)=1.
$$ 
The sequence $B_0=X_0\cup2X_0$ where $X_0$ is given 
in Example \ref{exa27}  satisfies $\ba(B_0)=0$ and  the above required conditions.

\subsection{} Let $\theta\in\R\smallsetminus\Q$. For $0<\alpha<1$, let 
$$
A_{\alpha}=\{n\in\N\,:\, \{\theta n\}<\alpha\},
$$
where $\{u\}=u-\lfloor u\rfloor$ is the fractional part of the real number $u$.
Since $(\theta n)_{n\in\N}$ is equidistributed modulo $1$,  we see that $A_{\alpha}$ has asymptotic density $\da(A_{\alpha})=\alpha$. 
 The following results can also be deduced from this.

\begin{prop}\label{pp51}
Let $0<\alpha\le \beta<1$.
We have 
\begin{enumerate}[itemsep=2pt,parsep=2pt]

\item $|A_{\alpha}\cap[k,k+(n-1)q]|=(\alpha+o(1)) n, \quad n\to\infty,$
uniformly in $k$ and $q\ge1$.

\item  $\ua(A_{\alpha})=\alpha$.

\item  $\ua(A_{\alpha}+A_{\beta})=\min(1,\alpha+\beta)$.

\item  $\bup(A_{\alpha})=1$ and $\bdo(A_{\alpha})=0$.

\item $\bup(A_{\alpha}+A_{\beta})=1$ and $\alpha+\beta<1$ implies
$\bdo(A_{\alpha}+A_{\beta})=0$.

\item If $0<\alpha+\beta<1$, there is no integer $q\ge1$ such that \eqref{eqbdbd} is satisfied 
with $A=A_{\alpha}$ and $B=A_{\beta}$.

\end{enumerate}
\end{prop}

\begin{proof}
\begin{enumerate}[itemsep=2pt,parsep=2pt]
\item follows also from the equidistribution modulo $1$ of $(\theta n)_{n\in\N}$.

\item follows from i) by letting $q=1$.

\item If $\alpha+\beta>1$ then 
we obtain from i) that by considering $\varepsilon<(\alpha+\beta-1)/2$, for any large enough $n$, we have
$|A_{\alpha}\cap[1,n]|>(\alpha-\varepsilon)n$ and 
$|(n-A_{\beta})\cap[1,n]|>(\beta-\varepsilon)n$ thus
$|A_{\alpha}\cap[1,n]|+|(n-A_{\beta})\cap[1,n]|>n$ giving
 $A_{\alpha}\cap(n-A_{\beta})\cap[1,n]\ne \varnothing$. This implies that
any large enough $n$ is in $A_{\alpha}+A_{\beta}$ hence
$\N\smallsetminus(A_{\alpha}+A_{\beta})$ is finite yielding 
$\ua(A_{\alpha}+A_{\beta})=1$.\\[0.3em]
Assume now $0<\alpha+\beta\le 1$.
We imitate proof of Lemma 1.4 of \cite{BH}.
The upper bound $\uup(A_{\alpha}+A_{\beta})\le \ua(A_{\alpha+\beta})=\alpha+\beta$ follows from the inclusion relation $A_{\alpha}+A_{\beta}\subset A_{\alpha+\beta}$ and  ii).
Let $0<\varepsilon<\alpha$. Again by equidistribution modulo $1$ of $(\theta n)_{n\in\N}$  there exists $C(\varepsilon)>0$ such that any open interval contained in $]0,1[$ of Lebesgue measure $\varepsilon$ contains  the fractional part $\{\theta m\}$ for some positive integer $m<C(\varepsilon)$. Let $n\ge C(\varepsilon)$ such that $x:=\{\theta n\}\in]\varepsilon,\alpha+\beta-\varepsilon[$. From above there exists $m<C(\varepsilon)$ such that $\{\theta m\}\in]\beta(x-\varepsilon)/(\alpha+\beta),\alpha(x+\varepsilon)/(\alpha+\beta)[\subset]0,\alpha[$. We thus have $m\in A_{\alpha}$, $n-m>0$ and 
$$
\{\theta(n-m)\}=\{\theta n-\theta m\}=\{\theta n\}-\{\theta m\}
\in\left]\frac{\beta x-\alpha\varepsilon}{\alpha+\beta},\frac{\alpha x+\beta\varepsilon}{\alpha+\beta}\right[\subset]0,\beta[.
$$
since $\alpha\le \beta$ and $\varepsilon<x<\alpha+\beta-\varepsilon$.
We infer that $n-m\in A_{\beta}$ and $n\in A_{\alpha}+A_{\beta}$. This
gives 
$$
(A_{\alpha+\beta-\varepsilon}\smallsetminus A_{\varepsilon})\cap[C(\varepsilon),\infty[ \subset A_{\alpha}+A_{\beta}.
$$
This implies $\udo(A_{\alpha}+A_{\beta})\ge \udo(A_{\alpha+\beta-\varepsilon})-\uup(A_{\varepsilon})=\alpha+\beta-2\varepsilon$. We then let $\varepsilon$ tending to $0$ and get $\udo(A_{\alpha}+A_{\beta})\ge \alpha+\beta$, as required.

\item follows from i)

\item is obvious by iv) and double inclusion $A_{\alpha}\subset A_{\alpha}+A_{\beta}\subset A_{\alpha+\beta}$.

\item  Take $q$ such that \eqref{eqbdbd} is satisfied with $A=B=A_{\alpha}$.  Then by i) 
we must have $S=\{0,1,\dots,q-1\}$ hence 
$\uup(A_{\alpha}+A_{\alpha})=1$, a contradiction to iii).
\end{enumerate}
\end{proof}
We thus have $\bup(A_{\alpha}+A_{\alpha})<2\bup(A_{\alpha})$ hence 
structural Theorem \ref{thmsec} applies whereas vi) of Proposition \ref{pp51}
shows that conclusions of structural Theorem 1.4 of \cite{Jin2} are not fulfilled.
Thus when considering Buck density, Theorems \ref{thmprinc} and \ref{thmsec} may possibly provide modular information on the structure of sumsets which is not detected 
when considering instead upper Banach density.

\section{\bf  Further results, constructions and concluding remarks}

\subsection{Sparse periodicity}
Assertions iv) of Theorem \ref{thmprinc} and v) of Theorem \ref{thmsec} highlights a kind of \emph{sparse periodicity} for $S=\sum_{i=1}^kX_i$ and $S:=X+X$ respectively. Indeed this
can be shortly written
$$
\forall m\ge1,\quad S^{(mq)}=S^{(q)}+q\{0,1,\dots,m-1\},
$$
while usual periodicity means $S=S^{(q)}+q\N$.

\subsection{Kneser's theorem with lower Buck density}\label{SS62}
In view of Kneser's theorem and Theorem \ref{thmprinc}, we may expect that when $\bdo(X_1+\cdots+X_k)<\sum_{i=1}^k\bdo(X)$ the sumset
$X_1+\cdots+X_k$ should be periodic. It can be shown that a similar result
to Theorem \ref{thmprinc} holds true through replacing $\bup$ by $\bdo$ and Statement iv)
by \\
\emph{
for all integers $m\ge1$
and all $(k+1)$-tuples  $(i_1,\dots,i_k,h)\in\prod_{j=1}^k\{1,\dots,r_j\}\times\{0,\dots,m-1\}$, $(a_{1i_1}+\cdots + a_{ki_k}+hq+ mq\N) \smallsetminus(X_1+\cdots+X_k) \text{ is finite}$.
}

\subsection{A startling construction}\label{SS63}
In the following we revisite Example \ref{exa27} and a result due to Leonetti and Tringali in \cite{Leo2}.

Let $K=(k_t)_{t\ge0}$ be an infinite increasing sequence of nonnegative integers
and 
$$
D_{K}:=\left\{\sum_{f\in F}2^f\,:\, F\subset\N,\ |F|<\infty\text{ and }F\cap K=\varnothing \right\}.
$$
Then for any $t\ge0$, $D_K\subset S_t+2^{k_t}\N$ where
$S_t$ is a finite set of cardinality $2^{k_t-t}$. It follows that 
$\ba(D_K)=\bup(D_K)=0$. Set for convenience $k_{-1}=-1$. Hence if  
$\sum_{j\ge0}x_j2^j\in D_K+D_K$, where $\forall j,\ x_j\in\{0,1\}$, then
for any integer $t\ge-1$ there exists $k_{t}<j\le k_{t+1}$ such that $x_j=0$. The converse  is also true. For $t\ge0$, let 
$$
M_t=\sum_{j=k_{t-1}+1}^{k_{t}}2^j=2^{k_t+1}-2^{k_{t-1}+1}.
$$ 
Obviously $2^{k_{t-1}+1}\mid M_t$ and $(M_t)_{t\ge0}$ is increasing.
 
We thus have
$$
E_K:=\N\smallsetminus(D_K+D_K)=\bigcup_{t\ge0}X_t,
$$
where $X_t:=\{0,1,\dots,2^{k_{t-1}+1}-1\}+M_t+2^{k_t+1}\N$, $t\ge0$. On the one hand this implies that
$E_K$ contains arbitrarily long intervals, hence $\bup(E_K)=\uup(E_K)=1$. Therefore by \eqref{eqn-1}, we get $\bdo(D_K+D_K)=1-\bup(E_K)=0$.
On the other hand we will prove the following.

\begin{lem}\label{lem61}
One has $\bup(D_K+D_K)= \delta_K:=\prod_{t=0}^{\infty}\left(1-2^{k_{t-1}-k_t}\right)$.
\end{lem}

\begin{proof}
Let $T\ge0$ and set $E_{K,T}=\bigcup_{t=0}^TX_t$. We have
$E_{K,T}\in \ZZ$ and may write
$$
E_{K,T}=Z_T+2^{k_T+1}\N
$$
where $Z_T\subset\{0,\dots,2^{k_T+1}-1\}$. The desired result will follow
from the identity
\begin{equation}\label{eq16}
\frac{|Z_T|}{2^{k_T}+1}=\xi_{K,T}:=1-\prod_{t=0}^{T}\left(
1-\frac{2^{k_{t-1}+1}}{2^{k_t+1}}
\right)
\end{equation}
and the plain fact that $(\xi_{K,T})_{T\ge0}$ is increasing.
To a certain extent this is consequence of the exclusion-inclusion principle. Indeed 
\eqref{eq16} is true when $T=0$ and 
we have the disjoint union
$$
Z_{T+1}=\left(\{M_{T+1},\dots,2^{k_{T+1}+1}-1\}\right)\sqcup
\left(Z_T+\{0,2^{k_{T}+1},2\cdot2^{k_{T}+1},\dots,M_{T+1}-2^{k_{T}+1}\}\right),
$$
hence
$$
|Z_{T+1}|=2^{k_{T+1}+1}-M_{T+1}+\frac{M_{T+1}}{2^{k_T+1}}|Z_T|
$$
giving
$$
1-\frac{|Z_{T+1}|}{2^{k_{T+1}+1}}=
\frac{M_{T+1}}{2^{k_{T+1}+1}}\left(1-\frac{|Z_T|}{2^{k_T+1}}\right).
$$
We thus have \eqref{eq16} by arguing inductively.
Since $E_{K,T}\subset E_K$, we firstly infer that $\bdo(E_K)\ge \bdo(E_{K,T})=\da(E_{K,T})=\xi_{K,T}$ for any $T\ge0$. Secondly we have for any $T\ge0$
$$
|E_K\cap[0,2^{k_T+1}-1]|=|Z_T|,
$$
hence by \eqref{eqn0} and \eqref{eq16}, $\bdo(E_K)\le \ddo(E_K)\le \lim_{T\to\infty}\xi_{K,T}=1-\delta_K$. We conclude by \eqref{eqn-1}.
\end{proof}

For any $0\le \gamma<1$ we can find an appropriate sequence $K=(k_t)_{t\ge0}$ so that 
$\delta_K=\gamma$. Thus by letting $A=D_K$ we deduce the following result
(see \cite[Theorem 3.3]{Leo2}).

\begin{prop}\label{pp62}
For any real number  $\gamma\in[0,1)$ there exists a sequence $A\subset\N$ such
that $\ba(A)=0$, $\bdo(A+A)=0$ and $\bup(A+A)=\gamma$. 
\end{prop}

\begin{rem}\label{rem63}
\begin{enumerate}[itemsep=2pt,parsep=2pt]

\item In fact the above arguments show that the above proposition remains true with uniform density instead of Buck density in the three places.

\item Note that from the proof of Lemma \ref{lem61} we may see that $\da(D_K+D_K)=\delta_K$ exists. Therefore by Lemma \ref{lem61} and \eqref{eqn0} this also yields $\ua(D_K+D_K)=\delta_k$.

\end{enumerate}
\end{rem}

\subsection{Completing Proposition \ref{pp62}}\label{SS63b}
In order to round off Proposition \ref{pp62} one can study the existence of a sequence $B$ such that
$\ba(B)=0$, $\bdo(B+B)=0$ and $\bup(B+B)=1$. 

\medskip
We start by two lemmas. The second one follows from the first which concerns thin finite additive basis in $\Z/m\Z$, $m\ge2$. 
For our need we shall not use the best known asymptotic bound due to Kohonen (cf. \cite{Koh})
which asserts that for any large enough $m$ there exists a set $A\subset\N$ such that $|A|\le \sqrt{\frac{294m}{85}}$  and $\{0,1,\dots,m-1\}\subset A+A$. We state and give a proof of the 
following weak bound which improves very slightly that of \cite{Be-He}, namely
$|A|<2\sqrt{m-1}$ whenever $m>2$.

\begin{lem}\label{lem64}
Let $m\ge2$ be an integer. Then there exists a set $A_m\subset\{0,1,\dots,m-1\}$ such that
$0\in A_m$, $\varphi_m(A_m+A_m)=\Z/m\Z$ and
\begin{equation}\label{eqam}
|A_m| \le 2\lfloor\sqrt{m+0.25}-0.5\rfloor< 2\sqrt{m-\sqrt{m}+0.5}<2\sqrt{m}.
\end{equation}
\end{lem}
\begin{proof}
Let $q=\lfloor\sqrt{m}\rfloor$ and $
A_m=\{0,1,\dots,s\}\cup\{2s+1,3s+2,\dots,qs+q-1\}$ with 
$$
s=\begin{cases} q-1 &\text{ if $q^2\le m < q(q+1)$,}\\
q&\text{ if $q(q+1)\le m < (q+1)^2$.}
\end{cases}
$$
Then $\{0,1,\dots,m-1\}\subset A_m+A_m$ since $(q+1)s+q-1\ge m-1$ in both cases.
Furthermore we get $|A_m|=2q-1$ or $2q$ according to whether $q^2\le  m<q(q+1)$ or $q(q+1)\le m<(q+1)^2$. We infer $|A_m|\le 2\lfloor\sqrt{m+0.25}-0.5\rfloor$ in the latter case and
$|A_m|\le 2\lfloor\sqrt{m}\rfloor-1\le 2\lfloor\sqrt{m}-0.5\rfloor\le 2\lfloor\sqrt{m+0.25}-0.5\rfloor $ in the former.

The second inequality in \eqref{eqam} follows from 
$
m-\sqrt{m}+0.5>m+0.25-\sqrt{m+0.25}+0.25  
=
(\sqrt{m+0.25}-0.5)^2.
$
\end{proof}

\begin{lem}\label{lem65}
Let $m,n\ge2$ be two arbitrary integers and denote $B=A_m+mA_n$. Then 
$|B|=|A_m||A_n|<4\sqrt{mn}$ and $\varphi_{mn}(B+B)=\Z/mn\Z$.
\end{lem} 
\begin{proof}
Clear from Lemma \ref{lem64}.
\end{proof}

Let $(m_k)_{k\ge1}$ be a sequence of integers greater than $1$.
By applying repeatedly the above lemma  the sequence
$$
B_k:=\sum_{i=1}^k\left(\prod_{1\le j\le  i-1}m_j\right)A_{m_i}=
A_{m_1}+m_1A_{m_2}+m_1m_2A_{m_3}+\cdots+ m_1\dots m_{k-1}A_{m_k}
$$
satisfies $|B_k|<2^k\sqrt{m_1\dots m_k}$ and $\varphi_{m_1\dots m_k}(B_k+B_k)=\Z/m_1\dots m_k\Z$. The sequence $(B_k)_{k\ge1}$ is increasing for the inclusion of sets, hence we can define
$$
B=\sum_{i=1}^{\infty}\left(\prod_{1\le j\le  i-1}m_j\right)A_{m_i}=\bigcup_{k\ge1}B_k.
$$
For any $k$ we have $\varphi_{m_1\dots m_k}(B)=\varphi_{m_1\dots m_k}(B_k)$ thus
\begin{equation}\label{eqbupB}
\bup(B)\le \frac{2^k}{\sqrt{m_1\dots m_k}}.
\end{equation}
We now fix any sequence $(m_k)$ such that for any integer $m$ there exists $k$ such that
$m|m_1\dots m_k$. This implies that the right-hand-side term of \eqref{eqbupB} tends to $0$ when
$k\to\infty$. Thus $\ba(B)=0$. Let $m$ be a positive integer and $k$ such that $m|m_1\dots m_k$.
Since $\varphi_{m_1\dots m_k}(B+B)=\varphi_{m_1\dots m_k}(B_k+B_k)=\Z/m_1\dots m_k\Z$ we also have $\varphi_m(B+B)=\Z/m\Z$. We infer $\bup(B+B)=1$.

Finally for any $k\ge1$ changing each element $a\ne0$ of $A_{m_k}$ by $a+\lambda_am_k$
for enough large $\lambda_a\in\N$ 
we can easily make $B$ very sparse in such a way that $\ua(B+B)=0$. This implies
$\bdo(B+B)=0$ by \eqref{eqn0}. Note also that distending $B$ by the above process does not affect $\varphi_m(B)$ and the crucial fact that $\varphi_m(B+B)=\Z/m\Z$.

\medskip
We thus have the following facts which was firstly observed 
in \cite{Leo2} by a different construction.

\begin{prop}\label{pp62b}
For an appropriate choice of the sequence $(m_k)_{k\ge1}$, 
the set $B$ satisfies 
$\ba(B)=0$, $\bdo(B+B)=\ua(B+B)=0$ and $\bup(B+B)=1$. 
\end{prop}

\subsection{$A$ has a positive Buck density does not imply that $A+A$ has a Buck density}
\label{SS65}
Proposition 2.2 of  \cite{HHP} provides a sequence $A$ having a positive asymptotic density $\da(A)=\frac37$ while $\dup(A+A)=1>\frac67= \ddo(A+A)$.  
It is given by
$$
A=\{0,1,2,3\}\cup\bigcup_{k\ge1}\Big((\{0,2,3\}+7\N)\cap[7N_{2k},7N_{2k+1}]\cup
(\{0,1,2\}+7\N)\cap[7N_{2k+1},7N_{2k+2}]\Big)
$$
where $(N_k)$ is any sequence of positive integers satisfying $N_{k}=o(N_{k+1})$, $k\to\infty$. We easily find that $\bdo(A)=\frac27$ and $\bup(A)=\frac47$. We ask the following question.\\[0.2em]
\textbf{Question.} \emph{Let $0<\alpha<1$. Does there exist a set $A\subset \N$ with positive Buck density $\ba(A)=\alpha$ and verifying
$\bdo(A+A)\ne \bup(A+A)$ ?}
\\[0.5em]
When $\alpha >\frac12$ such a set $A$ does not exist since otherwise we would have
$\ddo(A+A)\le 1<2\ba(A)=2\da(A)=2\ddo(A)$ by \eqref{eqn0}. By Kneser's theorem this would imply that $A+A$ is periodic and has a Buck density, a contradiction. When $\alpha=\frac12$ we get a similar contradiction by using 
 \S\ref{SS62} instead, that is an appropriate version of Kneser's theorem with lower Buck density.
\\[0.5em]
However we can answer affirmatively when $0<\alpha<\frac14$. 
This will depend on a construction which combines both  sequence $B_{\alpha}$ of Example \ref{exa28} and 
sequence $D_K$ of  \S \ref{SS63}.\\[0.3em]
Let $A=A_{\alpha,K}=B_{\alpha}\cup D_K$. Since $\ba(D_K)=0$ and $\ba(B_{\alpha})=\alpha$ we firstly have
$\ba(A)=\alpha$ by  i) of Lemma \ref{lem21}. Secondly $D_K+D_K\subset A+A$ hence
$\bup(A+A)\ge \delta_K$ by Lemma \ref{lem61}. Thirdly
since $\bdo(D_K+D_K)=0$ we have by iii) of Lemma \ref{lem21}
\begin{equation}\label{bdoAA}
\bdo(A+A)\le \bup((B_{\alpha}+D_K)\cup(B_{\alpha}+B_{\alpha})).
\end{equation}
Let $r=\min\{n\ge0 : \lfloor 2^n\alpha\rfloor=1\bmod 2\}$, assume that $\{j\ge0 : k_j\le r-2\}$ is non empty and let $s=1+\max\{j\ge0 : k_j\le r-2\}$. By \eqref{Balpha} we get
$$
(B_{\alpha}+B_{\alpha})\cup (B_{\alpha}+D_K)\subset
\left\{\sum_{f\in F}2^f\,:\, F\subset\N,\ |F|<\infty\text{ and }F\cap \{k_0,k_1,\dots,k_{s-1}\}=\varnothing \right\}.
$$
The sequence defined by the right-hand side member is periodic in $\N$  and has Buck density $2^{-s}$. By \eqref{bdoAA} we infer $\bdo(A+A)\le 2^{-s}$. If $\alpha<\frac14$ we have $r\ge3$ and we can fix $k_0=1$. Thus 
$s\ge1$ and $\bdo(A+A)\le\frac12$. By an appropriate choice of $(k_t)_{t\ge1}$, we can make sure that 
$\prod_{t=0}^{\infty}\left(1-2^{k_{t-1}-k_t}\right)>\frac12$. \\[0.3em]
We can state our result which completes Proposition \ref{pp62} and gives Theorem \ref{thm13}.

\begin{prop}\label{pp67}
Let $0<\alpha<\frac14$. Then there exists a set $A\subset \N$ such that
$\ba(A)=\alpha$ and $\bdo(A+A)\le \frac12<\bup(A+A)$.
\end{prop}

\subsection{On a question of Ruzsa}
Let $A$ be such that $\alpha:=\bdo(A)>0$. Let $\varepsilon >0$.
Then  there exist
a positive integer $q$ and two sets $R,S\subset\{0,1,\dots,q-1\}$,
such that 
\begin{align*}
&R+q\N\subset A\subset S+q\N,\\
&(\alpha-\varepsilon)q\le |R|\le \alpha q,\\
&\forall s\in S,\ A\cap(s+q\N)\ne\varnothing.
\end{align*}
We still denote $R$ and $S$ their images by the natural homomorphism from $\Z$ onto
$\Z/q\Z$. Therefore
$$
\frac{|R+S|}q\le \bdo(A+A)\le \bup(A+A)\le \frac{|S+S|}q.
$$
By a consequence of Pl\"unnecke's inequality (see Theorem 1.8.7 of \cite{Ruz} or Corollary 6.28 of \cite{TV}) applied in the abelian group $\Z/q\Z$, we get $|R||S+S|\le|R+S|^2$, thus
$$
(\alpha-\varepsilon)\bup(A+A)\le \bdo(A+A)^2.
$$
By letting $\varepsilon$ tend to $0$ we get the following statement.

\begin{prop}\label{pp68}
Let $A\subset \N$. Then
\begin{equation}\label{eqbdo}
\bdo(A+A)\ge \sqrt{\bdo(A)\bup(A+A)}.
\end{equation}
\end{prop}

As a corollary we give a partial answer to a question of Ruzsa involving asymptotic density.

\begin{cor}\label{cor69}
Let $A\subset \N$. If $A$ admits a Buck  density then
$$
\ddo(A+A)\ge \sqrt{\da(A)\dup(A+A)}.
$$
\end{cor}

\begin{proof}
By Lemma \ref{lem22} we have $\da(A)=\ba(A)$, $\ddo(A+A)\ge\bdo(A+A)$
and $\dup(A+A)\le \bup(A+A)$. Hence the result by \eqref{eqbdo}.
\end{proof}
 
\begin{rem}

\begin{enumerate}[itemsep=2pt,parsep=2pt]

\item A similar inequality holds true with uniform density under the same assumption:
$\udo(A+A)\ge \sqrt{\ua(A)\uup(A+A)}$. The proof is alike.

\item
Both for ensuring that Corollary \ref{cor69} is not trivial 
and for improving on Theorem \ref{thm13},  
a key step would be to obtain a sequence $A$ such that $\ba(A)$ exists and 
$\ddo(A+A)\ne \dup(A+A)$.

\end{enumerate}
\end{rem}


\begin{thebibliography}{99}

\bibitem{Be-He} E.A. Bertram and M. Herzog,  On medium-size subgroups and bases of finite groups, \textit{J. Combin. Theory Ser. A} \textbf{57} (1991), 1--14. 

\bibitem{BH} P-Y. Bienvenu and F. Hennecart, On the density or measure of sets and their sumsets in the integers or the circle, \textit{J. Number Theory} \textbf{212} (2020), 285--310.

\bibitem{Bie} P-Y. Bienvenu and F. Hennecart, 
Kneser’s Theorem in $\sigma$-finite Abelian groups, \textit{Bull. Canad. Math.} \textbf{65} (2022), 936--942.



\bibitem{Bih} P. Bihani and R. Jin, Kneser’s theorem for upper Banach density, \textit{J. Th\'eor. Nombres Bordeaux} \textbf{18} (2006), 323--343.

\bibitem{Bor} G. Bordes, Sum-sets of small upper density, \textit{Acta Arith.} \textbf{119} (2005), 187--200.



\bibitem{Buc1} R.C. Buck, The measure theoretic approach to density, \textit{Amer. J. Math.} \textbf{68} (1946), No. 4, 560--580.


\bibitem{Buc2} R.C. Buck, Generalized asymptotic density, \textit{Amer. J. Math.} \textbf{75} (1953), No. 2, 335--346.








\bibitem{Fre} G.A. Freiman, \textit{Foundations of a Structural Theory of Set Addition}, Transl. Math. Monogr., vol. 37, American Mathematical Society, Providence, RI, 1973, translated from the Russian.

\bibitem{Gri1} J.T. Griesmer, An inverse theorem: when the measure of the sumset is the sum of the measures in a locally compact
abelian group, \textit{Trans. Amer. Math. Soc.} \textbf{366} (2014), 1797--1827.


\bibitem{Gri} J.T. Griesmer, Small-sum pairs for upper Banach density in countable Abelian groups;
\textit{Adv. in Math.} \textbf{246} (2013), 220--264.

\bibitem{Gry1} D.J. Grynkiewicz, Quasi-periodic decompositions and the Kemperman structure theorem, \textit{European J. Combin.} \textbf{26}
(2005), 559--575.

\bibitem{Gry2} D.J. Grynkiewicz, A step beyond Kemperman’s structure theorem, \textit{Mathematika} \textbf{55} (2009), 67--114.


\bibitem{Hal} H. Halberstam and K.F. Roth, \textit{Sequences}, 2nd printing, Springer-Verlag, 1983.

\bibitem{HHP} N. Hegyv\'ary, F. Hennecart and P.P. Pach,
On the density of sumsets and product sets, 
\textit{Austral. J. Comb.} \textbf{74} (2019), 1--16.

\bibitem{IPT} M.R. Iac\`{o}, M. Pa\v{s}t\'{e}ka and R.F. Tichy, Measure density for set decompositions and uniform distribution, \textit{Rend. Circ. Mat. Palermo} \textbf{64} (2015), 323--339.

 \bibitem{Jin1} R. Jin, Solution to the Inverse problem for upper asymptotic density \textit{Journal für die reine und angewandte Mathematik} \textbf{595} (2006), 121--166.

\bibitem{Jin3} R. Jin, Freiman’s inverse problem with small doubling property, \textit{Adv. Math.}
\textbf{216} (2007), 711--752.

\bibitem{Jin2} R. Jin, Characterizing the structure of $A + B$ when $A + B$ has small upper Banach density, \textit{J. Number Theory} \textbf{130}
(2010), 1785--1800.





\bibitem{Kem} J.H.B. Kemperman, On small sumsets in an abelian group, \textit{Acta Math.}
\textbf{103} (1960), 63--88.

\bibitem{Kem2} J.H.B. Kemperman and P. Scherk, Complexes in Abelian groups, \textit{Canad. J. Math.} \textbf{6} (1954), 230--237.


\bibitem{Kne} M. Kneser, Absch\"atzung der asymptotischen Dichte von Summenmengen, \textit{Math. Z.} \textbf{58} (1953), 459--484. 

\bibitem{Kne2} M. Kneser, Summenmengen in lokalkompakten abelschen Gruppen, \textit{Math. Z.} \textbf{66} (1956), 88--110.












\bibitem{Koh} J. Kohonen, An improved lower bound for finite additive $2$-bases, \textit{J. Number Theory} \textbf{174} (2017), 518--524.


\bibitem{Leo} P. Leonetti and S. Tringali, On the notions of upper and lower density, \textit{Proc. Edinb. Math. Soc.} \textbf{63} (2020), No. 1,
139--167.

\bibitem{Leo2} P. Leonetti and S. Tringali, On small sets of integers, \textit{Ramanujan J.} \textbf{57} (2022), 275--289.

\bibitem{Leo3} P. Leonetti and S. Tringali, On the density of sumsets, \textit{Monatsh. Math.} \textbf{198} (2022), 565--580.

\bibitem{Leo4} P. Leonetti and S. Tringali, On the density of sumsets, II,  \textit{Bull. Austral. Math. Soc.}, \textbf{109} (2024), 414--419.


\bibitem{Pas} M. Pa\v{s}t\'eka,  Some properies of Buck's measure density, 
\textit{Math. Slovaca} \textbf{42} (1992), 15--32.

\bibitem{Pas2} M. Pa\v{s}t\'eka,  \textit{Four Approaches to density}
Spectr. Slovak. \textbf{3}
Peter Lang, Frankfurt am Main; VEDA, Publishing House of the Slovak Academy of Sciences, Bratislava, 2013, 97 pp.


\bibitem{Ruz} I.Z. Ruzsa, 
Sumsets and structure.
Adv. Courses Math. CRM Barcelona, Birkh\"auser Verlag, Basel, 2009, 87–210.

\bibitem{TV} T. Tao and  V.H. Vu, Additive combinatorics.  Cambridge Studies in Advanced Mathematics, 105. Cambridge University Press, Cambridge, 2006. xviii+512 pp.

\bibitem{Vos1} G. Vosper, The critical pair of subsets of a group of prime order, \textit{J. London Math. Soc.} \textbf{31} (1956), 200--205. Addendum, \textit{ibid.}, 280--282.


\end{thebibliography}
\end{document}